\documentclass[english, 12pt]{article}
\usepackage[T1]{fontenc}
\usepackage[latin9]{inputenc}
\usepackage{babel}
\usepackage{mathrsfs}
\usepackage{bm}
\usepackage{bbm}
\usepackage{amsmath}
\usepackage{amsthm}
\usepackage{appendix}
\usepackage[unicode=true,pdfusetitle,bookmarks=false,bookmarksnumbered=false,bookmarksopen=false,breaklinks=false,pdfborder={0 0 1},backref=false,colorlinks=true,allcolors=blue]{hyperref}

\makeatletter
\numberwithin{equation}{section}
\theoremstyle{plain}
\newtheorem{thm}{Theorem}[section]
\newtheorem{lemma}[thm]{Lemma}
\newtheorem{prop}[thm]{Proposition}
\newtheorem{cor}[thm]{Corollary}
\theoremstyle{definition}
\newtheorem{rmk}[thm]{Remark}

\usepackage{amssymb,amsfonts}
\usepackage{geometry}
\usepackage{algorithm}
\usepackage{algorithmic}
\usepackage{float}
\usepackage{amsthm}
\usepackage{amsmath}

\newcommand{\md}{\mathop{}\mathopen\mathrm{d}}

\makeatother

\begin{document}
\title{Uniform-in-Time Estimates on the Size of Chaos for Interacting Particle Systems}
\author{Pengzhi Xie}
\author{Pengzhi Xie\footnote{School of Mathematical Sciences, Fudan University, Shanghai 200433, China. 22110180046@m.fudan.edu.cn.}
}
\maketitle
\begin{abstract}
For any weakly interacting particle system with bounded kernel, we give uniform-in-time estimates of the $L^2$ norm of correlation functions, provided that the diffusion coefficient is large enough. When the condition on the kernels is more restrictive, we can remove the dependence of the lower bound for diffusion coefficient on the initial data and estimate the size of chaos in a weaker sense. Based on these estimates, we may study fluctuation around the mean-field limit.
\end{abstract}
\section{Introduction}\label{Intro}
\subsection{Overview}
Consider the first order interacting particle system on the torus with i.i.d. noise
\begin{equation}\label{ParticleSystem}
    \left\{
    \begin{aligned}
        &\md X_t^i=\frac{1}{N}\sum_{j=1}^NK(X_t^i,X_t^j)\md t+\sqrt{2\sigma}\md W_t^i,\\
        &X_t^i\vert_{t=0}=X_0^i
    \end{aligned}
    \right.
\end{equation}
where $X_t^i\in\mathbb{T}^d$ for $i=1,\dots,N$ are the positions of particles at time $t$. Here $W^i$'s are i.i.d. Brownian motions and $X_0^i\overset{i.i.d.}{\sim}\rho_0,\;i=1,\dots,N$. Then, at least formally, the joint distribution of $(X_t^1,\dots,X_t^N)$, denoted by $f_N$, satisfies the Liouville equation
\begin{equation}\label{Liouville}
    \left\{
    \begin{aligned}
        &\partial_t f_N+\sum_{i=1}^N\mathrm{div}_{x_i}\bigg( f_N\,\frac{1}{N}\sum_{j=1}^NK(x_i,x_j)\bigg)=\sigma\sum_{i=1}^N\Delta_{x_i}f_N,\\
        &f_N\vert_{t=0}=\rho_0^{\otimes N}.
    \end{aligned}
    \right.
\end{equation}
Then exchangeability of the particles is equivalent to symmetry of $f_N$ in its $N$ variables and we may study $f_{[j],N}$, the $j$-marginal of $f_N$, where $[j]=\{1,\dots,j\}$. That is,
\[f_{[j],N}(x_1,\dots,x_j)=\int\dots\int f_N(x_1,\dots,x_N)\md x_{j+1}\dots\md x_{N}.\]
Here and thereafter we put the index set on the subscript of some function $h$ to emphasize the variables involved in it. For any $P\subset\mathbb{N}\cup\{*\}$ with $\vert P\vert<\infty$, and any $h_P:(\mathbb{T}^d)^P\mapsto\mathbb{R}$, define
\begin{align}
    S_{k,l}h_P:(\mathbb{T}^d)^P&\mapsto\mathbb{R},\notag\\
    \Big(S_{k,l}h_P\Big)(x_\alpha;\alpha\in P)&:=\mathrm{div}_{x_k}\bigg(K(x_k,x_l)h_P(x_\alpha;\alpha\in P)\bigg).\notag
\end{align}
Moreover, if $*,k\in P$ and $k\ne *$, denote
\begin{align}
    H_kh_P:(\mathbb{T}^d)^{P- \{ * \}}&\mapsto\mathbb{R},\notag\\
    \Big(H_kh_P\Big)(x_\alpha;\alpha\in P- \{ * \})&:=\mathrm{div}_{x_k}\left(\int K(x_k,x_*)h(x_\alpha;\alpha\in P)\md x_*\right).\notag
\end{align}
Then the evolution of $f_{[j],N}$'s is described by the BBGKY hierarchy
\begin{align}\label{BBGKY}
    \partial_tf_{[j],N}-\sigma\sum_{k=1}^j\Delta_{x_k}f_{[j],N}+\frac{1}{N}\sum_{k,l=1}^jS_{k,l}f_{[j],N}=-\frac{N-j}{N}\sum_{k=1}^jH_kf_{[j]\cup\{*\},N}.
\end{align}
When the number of particles $N$ is large, we can seek for averaged behavior of one typical particle in the system. For example, consider the $1$-marginal $f_{[1],N}$, one may expect that it is close to the solution to the McKean-Vlasov equation
\begin{align}\label{McKeanVlasov}
    \partial_t\rho-\sigma\Delta\rho+\mathrm{div}\bigg(\rho\int_{\mathbb{T}^d} K(x,x_*)\rho(x_*)\md x_*\bigg)=0
\end{align}
with initial data $\rho_t\vert_{t=0}=\rho_0$.\\
Correction to such a mean-field approximation is usually based on the so-called correlation functions $g_{[j],N}$\cite{D21,HR23,BD24}, see also \cite{PS17,PPS19,BDJ24,DHM24} for its variants. Roughly speaking, $g_{[1],N}$ equals exactly the $1$-marginal $f_{[1],N}$ and all the $g_{[j],N}$'s are small for $j\ge2$. That is, we want to estimate $g_{[j],N}$ to determine the \textbf{size of chaos}. It is proved in \cite{D21,BD24} that for second order particle systems, the correlation function with $j$ variables is of order $\frac{1}{N^{j-1}}$. However, the kernel there is required to be smooth enough. The paper \cite{HR23} tried to conduct asymptotic expansion of the correlation functions and consider kind of $L^2$ bounds of them on finite time intervals, for first order diffusive particle system with bounded kernels. However, it is not clear whether such asymptotic expansion converges.
\subsection{Main Results}
Given all the marginals $f_{[j],N}$'s, define the $m$ particle correlation function by
\[g_{[m],N}:=\sum_{\pi\vdash[m]}(-1)^{\vert\pi\vert-1}(\vert\pi\vert-1)!\prod_{P\in\pi}f_{P,N}.\]
Thanks to the combinatorial identity
\begin{equation}
    \sum_{\sigma\le\pi}(-1)^{\vert\sigma\vert-1}(\vert\sigma\vert-1)!=\left\{\begin{aligned}
        &1,\quad&\vert\pi\vert=1,\\
        &0,\quad&\vert\pi\vert\ge2,
    \end{aligned}\right.
\end{equation}
we can express $f_{[j],N}$'s in terms of $g_{[m],N}$'s
\[f_{[j],N}=\sum_{\pi\vdash[j]}\prod_{P\in\pi}g_{P,N},\]
where $\pi\vdash[m]$ means that $\pi$ is a partition of the set $[m]$ and $\sigma\le\pi$ means that $\pi$ is a refinement of the partition $\sigma$. Our first result shows that, under the same setting as \cite{HR23}, we can decide the exact $L^2$ size of chaos uniformly-in-time.
\begin{thm}\label{MainThm1}
    Assume that the kernel $K\in L^\infty(\mathbb{T}^{2d};\mathbb{R}^d)$ and that $\rho_0\in L^2(\mathbb{T}^d)$. Let $f_{[j],N}$'s be solution to BBGKY hierarchy \eqref{BBGKY} with chaotic initial data and let $g_{[m],N}$'s be the corresponding correlation functions. Then there exist constants $C=C(\Vert K\Vert_{L^\infty},\Vert\rho_0\Vert_{L^2})$ and $C_0=C(\Vert\rho_0\Vert_{L^2})$ such that, if $\sigma>C$, then for all $N$ the following quantitative estimate of the $L^2$ size of chaos
    \[\Vert g_{[m],N}(t)\Vert_{L^2(\mathbb{T}^{dm})}\le\frac{C_0(m-1)!}{m^2N^{m-1}}\]
    holds for all $m=1,\dots,N$ and all $t\ge0$.
\end{thm}
\begin{rmk}
    If we estimate, e.g., the $L^2$ size of chaos properly, then we can derive propagation of chaos provided that the stability of McKean-Vlasov equation \eqref{McKeanVlasov} has been established, see Section \ref{Discussion}.
\end{rmk}
\begin{rmk}
    Theorem \ref{MainThm1} is essentially a result for \textbf{small initial data}, as the lower bound of $\sigma$ depends on $\rho_0$. Such assumption on $\sigma$ appears in \cite{LL23} as well as in Section 8 of \cite{D21}, and is stronger than that in \cite{BD24}.
\end{rmk}
To overcome the drawback mentioned above, we have to consider some framework weaker that $L^1$ since we have the apparent bound $\Vert\rho_0\Vert_1=1$ and thus we can replace the constant $C_0$ in \ref{MainThm1} by some universal constant, say $10$. See Section \ref{MainProof1} and Section \ref{MainProof2}.\\ Recall that propagation of chaos means that for any $j$, for any time $t$, the $j$-marginal of the joint distribution converges weakly to $\rho^{\otimes j}(t)$. As its name suggests, converging weakly is weaker than converging in $L^1$ norm, thus we need to find some norm compatible to weak convergence. There are several equivalent conditions for weak convergence of probability measures on $\mathbb{R}^D$ or $\mathbb{T}^D$. By definition, we say a sequence of probability measures $\mu_n$ on $\mathbb{D}=\mathbb{R}^D,\mathbb{T}^D$ converges weakly to some $\mu\in\mathcal{P}(\mathbb{D})$ if for any $\varphi\in C_c(\mathbb{D})$,
\begin{align}\label{weak}
    \int\varphi\md \mu_n\to\int\varphi\md \mu,\;n\to\infty.
\end{align}
When we take the domain $\mathbb{D}$ to be the torus $\mathbb{T}^D$, by Stone-Weierstrass theorem, the trigonometric polynomials, $\mathrm{span}\{e^{2\pi ik\cdot x};k\in\mathbb{Z}^D\}$, are dense in $C_c(\mathbb{T}^D)$. Thus we only need to check \eqref{weak} for every $\varphi\in\mathrm{span}\{e^{2\pi ik\cdot x};k\in\mathbb{Z}^D\}$. In other words, weak convergence is equivalent to pointwise convergence of Fourier modes for probability distributions on the torus $\mathbb{T}^D$. Thus the norm we ought to study is some weighted $l^\infty$ norm of Fourier coefficients.\\
We consider the original $l^\infty$ norm in this paper. For a measurable function $h$ on $\mathbb{T}^D$, define
\[\Vert h\Vert_{\hat{l}^\infty}:=\Vert\hat{h}\Vert_{l^\infty}=\sup_{k\in\mathbb{Z}^D}\vert\hat{h}(k)\vert,\]
and our second result concerns on the size of chaos under $\hat{l}^\infty$ norm.
\begin{thm}\label{MainThm2}
    Assume that the kernel $K$ satisfies that the Fourier modes of $K(x,y)$ is summable, i.e.
    \[\Vert\hat{K}\Vert_{l^1}:=\sum_{\eta,\lambda}\vert\hat{K}(\eta,\lambda)\vert<+\infty\] 
    and that $\rho_0\in L^2(\mathbb{T}^d)$. With the same notations as in Theorem \ref{MainThm1}, there exists a constant $C=C(K)$ such that, if $\sigma>C$, then for all $N$ the following quantitative estimate of the $\hat{l}^\infty$ size of chaos
    \[\Vert g_{[m],N}(t)\Vert_{\hat{l}^\infty(\mathbb{T}^{dm})}\le\frac{2(m-1)!}{m^2N^{m-1}}\]
    holds for all $m=1,\dots,N$ and all $t\ge0$.
\end{thm}
With such kind of size of chaos as in Theorem \ref{MainThm2}, we may prove a central limit theorem for empirical measure of the system \eqref{ParticleSystem}. Propagation of chaos implies that $X_i$'s are ``almost'' independent, so we may expect that they behaves like i.i.d. random variables. In particular, some central limit theorem would hold. Let
\[\mu_N^t:=\frac{1}{N}\sum_{i=1}^N\delta_{X_i^t}\]
be the empirical measure, then CLT means that for fixed test function $\phi$, the distributions of random variables
\[\int\phi\md \mu_N^t\]
approximates to some normal distribution. 
\begin{cor}[Uniform-in-Time CLT]\label{CLT}
    Under the assumption of Theorem \ref{MainThm2}, there exists a constant $C=C(K)$ such that for $\sigma>C$, for any test function $\phi\in C^\infty( \mathbb{T}^d)$, for any time $t$ such that the quantity
    \[\int\phi^2\md\rho_t-\left(\int\phi\md\rho_t\right)^2+\int\phi\otimes\phi\md b_t\]
    does not vanish, where $\rho_t$ is solution to the McKean-Vlasov equation \eqref{McKeanVlasov} with initial data $\rho_0$ and $b_t$ is defined in Section \ref{Discussion}, we have
    \begin{itemize}
        \item [i)~](Central Limit Theorem)There exists some constant $C_1$ depending on $\phi$ such that the random variable
        \[Y_N:=\frac{1}{\sigma_{\phi,N}}\Bigg(\int\phi\md\mu_N^t-\mathbb{E}\bigg(\int\phi\md\mu_N^t\bigg)\Bigg)\]
        satisfies
        \[\sup_{x\in\mathbb{R}}\vert\mathbb{P}(Y_N\le x)-\mathbb{P}(Z\le x)\vert\le\frac{C_1}{N^{1/7}},\]
        for $N$ large enough (independent on $t$), where $\sigma^2_{\phi,N}=Var\left(\int\phi\md\mu_N^t\right)$ and $Z\sim\mathcal{N}(0,1)$.
        \item [ii)~](Convergence of Variance)As $N\to\infty$, the variance $N\sigma^2_{\phi,N}$ converges to
        \[\int\phi^2\md\rho_t-\left(\int\phi\md\rho_t\right)^2+\int\phi\otimes\phi\md b_t\]
        which we mentioned above, uniformly in $t$.
    \end{itemize}
\end{cor}
\begin{rmk}
    In fact, for some first order systems with singular interaction, Central Limit Theorem can also be derived. Examples include the point-vertex dynamics \cite{JW18, WZZ23} and the equilibrium of Coulomb interactions \cite{S23}. Despite these results, we present a proof of the Corollary above to show that CLT for interacting particles systems would hold when the size of chaos were estimated properly.
\end{rmk}
\begin{rmk}
    To our best knowledge, there are many early works on Central Limit Theorem for interacting particle systems, such as \cite{BH77,M96,L09}, without providing convergence rate. The paper \cite{WZZ23} states that a CLT result similar to that in \cite{M96} holds when the relative entropy between the $N$-particle distribution and the $N$-tensor product of limit distribution is about $O(1/N)$, which has been justified for the point-vortex model in \cite{JW18}. In \cite{D21} and \cite{BD24}, CLT with convergence rate $O(1/N^{1/2})$ ($O(1/N^{1/3})$, resp.) was estabilshed for particle systems with smooth kernel. However, the kernels considered here are not smooth, so we can only get the rate $O(1/N^{1/7})$. 
\end{rmk}
\section{Preliminaries}\label{Preliminaries}
In this section we recall some preliminary results. Equivalent hierarchy concerning correlation functions has been derived in  \cite{HR23}. The utility of Fourier modes plays a key role in the proof of Theorem \ref{MainThm1}, so we list some basic notations here. Finally, we cite the Abstract Bootstrap Principle from \cite{T06}.
\subsection{Hierarchy for Correlation Functions}
Here and thereafter we shall omit the subscript $N$ of marginals and correlation functions. Since we can express marginals in terms of correlation functions and vice versa, we may get the evolution of $g_{[m]}$'s from BBGKY hierarchy \cite{HR23}
\begin{align}\label{gHierarchy}
    &\partial_tg_{[m]}-\sigma\sum_{k=1}^m\Delta_{x_k}g_{[m]}\notag\\=&-\frac{N-m}{N}\sum_{k=1}^mH_kg_{[m]\cup\{*\}}+\sum_{k=1}^m\sum_{W\subset[m]-\{k\}}\frac{m-1-\vert W\vert}{N}H_kg_{W\cup\{k,*\}}g_{[m]-\{k\}-W}\notag\\
    &-\frac{N-m}{N}\sum_{k=1}^m\sum_{W\subset[m]-\{k\}}H_kg_{W\cup\{k\}}g_{[m]\cup\{*\}-W-\{k\}}\notag\\
    &+\sum_{k=1}^m\sum_{W\subset[m]-\{k\}}\sum_{R\subset[m]-\{k\}-W}\frac{m-1-\vert W\vert-\vert R\vert}{N}H_kg_{W\cup\{k\}}g_{R\cup\{*\}}g_{[m]-R-W-\{k\}}\notag\\
    &-\frac{1}{N}\sum_{k,l=1}^mS_{k,l}g_{[m]}-\frac{1}{N}\sum_{k,l=1,k\ne l}^m\sum_{W\subset[m]-\{k,l\}}S_{k,l}g_{W\cup\{k\}}g_{[m]-\{k\}-W}.
\end{align}
In fact, the hierarchy \eqref{gHierarchy} is equivalent to BBGKY hierarchy \eqref{BBGKY}. The derivation of \eqref{gHierarchy} is omitted here.
\subsection{Fourier Transform on Torus}
Since all the variables $x_1,\dots,x_N$ take value on the torus $\mathbb{T}^d$, we may write any function $h_P$ in its Fourier modes
\[h_P(x_\alpha;\alpha\in P)=\sum_{\xi_\alpha\in\mathbb{Z}^d;\alpha\in P}\hat{h}_P(\xi_\alpha;\alpha\in P)e^{2\pi i\sum_{\alpha\in P}\xi_\alpha\cdot x_\alpha}.\]
Also, we may write the kernel $K$ as
\[K(x,y)=\sum_{\lambda,\eta\in\mathbb{Z}^d}\hat{K}(\lambda,\eta)e^{2\pi i(\lambda\cdot x+\eta\cdot y)}.\]
Then
\begin{align}
    &\Big(S_{k,l}h_P\Big)(x_\alpha;\alpha\in P)\notag\\
    =&\mathrm{div}_{x_k}\bigg(K(x_k,x_l)h_P(x_\alpha;\alpha\in P)\bigg)\notag\\
    =&\mathrm{div}_{x_k}\Bigg(\sum_{\lambda,\eta\in\mathbb{Z}^d}\hat{K}(\lambda,\eta)e^{2\pi i(\lambda\cdot x_k+\eta\cdot x_l)}\sum_{\xi_\alpha\in\mathbb{Z}^d;\alpha\in P}\hat{h}_P(\xi_\alpha;\alpha\in P)e^{2\pi i\sum_{\alpha\in P}\xi_\alpha\cdot x_\alpha}\Bigg)\notag\\
    =&\mathrm{div}_{x_k}\Bigg(\sum_{\xi_\alpha\in\mathbb{Z}^d;\alpha\in P}e^{2\pi i\sum_{\alpha\in P}\xi_\alpha\cdot x_\alpha}\sum_{\lambda,\eta\in\mathbb{Z}^d}\hat{K}(\lambda,\eta)\hat{h}_P(\xi_\alpha,\xi_k-\lambda,\xi_l-\eta;\alpha\in P-\{k,l\})\Bigg)\notag\\
    =&\sum_{\xi_\alpha\in\mathbb{Z}^d;\alpha\in P}e^{2\pi i\sum_{\alpha\in P}\xi_\alpha\cdot x_\alpha}(2\pi i\xi_k)\cdot\sum_{\lambda,\eta\in\mathbb{Z}^d}\hat{K}(\lambda,\eta)\hat{h}_P(\xi_\alpha,\xi_k-\lambda,\xi_l-\eta;\alpha\in P-\{k,l\})\notag
\end{align}
for $k\ne l$ and
\begin{align}
    &\Big(S_{k,k}h_P\Big)(x_\alpha;\alpha\in P)\notag\\
    =&\sum_{\xi_\alpha\in\mathbb{Z}^d;\alpha\in P}e^{2\pi i\sum_{\alpha\in P}\xi_\alpha\cdot x_\alpha}(2\pi i\xi_k)\cdot\sum_{\lambda,\eta\in\mathbb{Z}^d}\hat{K}(\lambda,\eta)\hat{h}_P(\xi_\alpha,\xi_k-\lambda-\eta;\alpha\in P-\{k\}).\notag
\end{align}
Similarly,
\begin{align}
    &\Big(H_kh_P\Big)(x_\alpha;\alpha\in P)\notag\\
    =&\mathrm{div}_{x_k}\Bigg(\int K(x_k,x_*)h(x_\alpha;\alpha\in P)\md x_*\Bigg)\notag\\
    =&\mathrm{div}_{x_k}\Bigg(\int\sum_{\lambda,\eta\in\mathbb{Z}^d}\hat{K}(\lambda,\eta)e^{2\pi i(\lambda\cdot x_k+\eta\cdot x_*)}\sum_{\xi_\alpha\in\mathbb{Z}^d;\alpha\in P}\hat{h}_P(\xi_\alpha;\alpha\in P)e^{2\pi i\sum_{\alpha\in P}\xi_\alpha\cdot x_\alpha}\md x_*\Bigg)\notag\\
    =&\mathrm{div}_{x_k}\Bigg(\sum_{\xi_\alpha;\alpha\in P-\{*\}}e^{2\pi i\sum_{\alpha\in P-\{*\}}\xi_\alpha\cdot x_\alpha}\sum_{\lambda,\eta\in\mathbb{Z}^d}\hat{K}(\lambda,\eta)\hat{h}_P(\xi_\alpha,\xi_k-\lambda,-\eta;\alpha\in P-\{*\})\Bigg)\notag\\
    =&\sum_{\xi_\alpha;\alpha\in P-\{*\}}e^{2\pi i\sum_{\alpha\in P-\{*\}}\xi_\alpha\cdot x_\alpha}(2\pi i\xi_k)\cdot\sum_{\lambda,\eta\in\mathbb{Z}^d}\hat{K}(\lambda,\eta)\hat{h}_P(\xi_\alpha,\xi_k-\lambda,-\eta;\alpha\in P-\{*\}).\notag
\end{align}
To represent the Fourier modes of $S_{k,l}h_P$ and $H_kh_P$, we may define
\[\hat{S}_{k,l}\varphi(\xi_\alpha;\alpha\in P):=\sum_{\lambda,\eta}(2\pi i\xi_k)\cdot\hat{K}(\lambda,\eta)\varphi(\xi_\alpha,\xi_k-\lambda,\xi_l-\eta;\alpha\in P-\{k,l\}),\]
\[\hat{S}_{k,k}\varphi(\xi_\alpha;\alpha\in P):=\sum_{\lambda,\eta}(2\pi i\xi_k)\cdot\hat{K}(\lambda,\eta)\varphi(\xi_\alpha,\xi_k-\lambda-\eta;\alpha\in P-\{k\}),\]
\[\hat{H}_k\varphi(\xi_\alpha;\alpha\in P-\{*\}):=\sum_{\lambda,\eta}(2\pi i\xi_k)\cdot\hat{K}(\lambda,\eta)\varphi(\xi_\alpha,\xi_k-\lambda,-\eta;\alpha\in P-\{k,*\}).\]
For any function $h_P$ satisfying
\[\int_{\mathbb{T}^d} h_P(x_\alpha;\alpha\in P)\md x_k=0\]
for some $k\in P$, we may define $\vert\nabla_k\vert^{-1}h_P$ by
\begin{equation}
    \bigg(\vert\nabla_k\vert^{-1}h_P\bigg)^\wedge(\xi_\alpha;\alpha\in P):=\left\{
    \begin{aligned}
        &\vert2\pi\xi_k\vert^{-1}\hat{h}_P(\xi_\alpha;\alpha\in P),\qquad&\xi_k\ne 0,\notag\\
        &0,\qquad&\xi_k=0.
    \end{aligned}
    \right.
\end{equation}
Then $\vert\nabla_k\vert^{-1}$ is a linear operator on certain subspace of $L^2((\mathbb{T}^{d})^P)$. In particular, we have
\begin{lemma}\label{BoundOp}
    Suppose $K\in L^\infty(\mathbb{T}^{2d};\mathbb{R}^d)$, then $\vert\nabla_k\vert^{-1}S_{k,l}$ ($\vert\nabla_k\vert^{-1}H_k$) is a bounded operator from $L^2((\mathbb{T}^{d})^P)$ to $L^2((\mathbb{T}^{d})^P)$ ($L^2((\mathbb{T}^{d})^{P-\{*\}})$, resp.) with operator norm no bigger than $\Vert K\Vert_{L^\infty}$, for any suitable index set $P$.
\end{lemma}
\begin{proof}
    We may prove the boundedness of $\vert\nabla_k\vert^{-1}H_k$. The case of $\vert\nabla_k\vert^{-1}S_{k,l}$ is similar. For any $P$ such that $*,k\in P$ and $k\ne *$, for any $h_P\in L^2((\mathbb{T}^d)^P)$, by Plancherel's identity,
    \begin{align}
        \Vert\vert\nabla_k\vert^{-1}H_kh_P\Vert_{L^{2}}^2=&\sum_{\xi_\alpha;\alpha\in P,\xi_k\ne 0}\left\vert\frac{1}{\vert2\pi\xi_k\vert}\hat{H}_k\hat{h}_P(\xi_\alpha;\alpha\in P-\{*\})\right\vert^2\notag\\
        \le&\sum_{\xi_\alpha;\alpha\in P,\xi_k\ne 0}\left\vert\sum_{\lambda,\eta}\hat{K}(\lambda,\eta)\varphi(\xi_\alpha,\xi_k-\lambda,-\eta;\alpha\in P-\{k,*\})\right\vert^2\notag\\
        \le&\sum_{\xi_\alpha;\alpha\in P}\left\vert\sum_{\lambda,\eta}\hat{K}(\lambda,\eta)\varphi(\xi_\alpha,\xi_k-\lambda,-\eta;\alpha\in P-\{k,*\})\right\vert^2\notag\\
        =&\int_{(\mathbb{T}^d)^{P-\{*\}}}\Bigg\vert\int_{\mathbb{T}^d}K(x_k,x_*)h_P(x_\alpha;\alpha\in P)\md x_*\Bigg\vert^2\notag\\
        \le&\Vert K\Vert_{L^\infty}^2\Vert h_P\Vert_{L^2}^2.\notag
    \end{align}
\end{proof}
\subsection{Bootstrap Method}
To deal with the nontrivial dependency in \eqref{gHierarchy}, we need the so-called bootstrap principle, which can be seen as a more general version of Gronwall's inequality. Here we present such principle in the form of Proposition 1.21, \cite{T06}.
\begin{prop}[Abstract Bootstrap Principle]\label{Bootstrap}
    Let $I$ be a time interval, and for each $t\in I$ we have two statements, a ``hypothesis'' $\mathbf{H}(t)$ and a ``conclusion'' $\mathbf{C}(t)$. Suppose we can verify the following four assertions:
    \begin{itemize}
        \item [(a)~](Hypothesis implies conclusion) If $\mathbf{H}(t)$ is true for some time $t\in I$, then $\mathbf{C}(t)$ is also true for that time $t$.
        \item [(b)~](Conclusion is stronger than hypothesis) If $\mathbf{C}(t)$ is true for some $t\in I$, then $\mathbf{H}(t')$ is true for all $t'\in I$ in a neighbourhood of $t$.
        \item [(c)~](Conclusion is closed) If $t_1,t_2,\dots$ is a sequence of times in $I$ which converges to another time $t\in I$ and $\mathbf{C}(t_n)$ is true for all $n$, then $\mathbf{C}(t)$ is true.
        \item [(d)~] (Base case) $\mathbf{H}(t)$ is true for at least one time $t\in I$.
    \end{itemize}
    Then $\mathbf{C}(t)$ is true for all $t\in I$.
\end{prop}
\begin{proof}
    Let $\Omega$ be the set of times $t\in I$ for which $\mathbf{C}(t)$ holds. Properties $(d)$ and $(a)$ ensure that $\Omega$ is non-empty. Properties $(b)$ and $(a)$ ensure that $\Omega$ is open. Property $(c)$ ensures that $\Omega$ is closed. Since the interval $I$ is connected, we thus see that $\Omega=I$, and the claim follows.
\end{proof}
\subsection{Cumulants and Normal Approximation}
In the proof of CLT in Section \ref{Discussion}, it is convenient to apply the method of cumulants (\cite{DJS22}, see also \cite{SS12}). For any random variable $X$ with its characteristic function smooth enough, define its $m$-th cumulant $\kappa_m(X)$ by
\[\kappa_m(X):=(-i)^m\frac{\md^m}{\md^mt}\log\mathbb{E}e^{itX}\bigg\vert_{t=0}.\]
In fact, we can express the $m$-th cumulant in terms of the $j$-th moments with $j\le m$.
\[\kappa_m(X)=\sum_{\pi\vdash[m]}(-1)^{\vert\pi\vert-1}(\vert\pi\vert-1)!\prod_{B\in\pi}\mathbb{E}[X^{\vert B\vert}].\]
Conversely,
\[\mathbb{E}[X^j]=\sum_{\pi\vdash[j]}\prod_{B\in\pi}\kappa_{\vert B\vert}(X).\]
It is easy to verify that for all $m\ge3$, the $m$-th cumulant of standard normal distribution vanishes. So for some standardized random variable $X$(i.e., $\mathbb{E}X=0,\mathrm{Var}(X)=1$) with decaying cumulants, we may expect that the distribution of $X$ is close to $\mathcal{N}(0,1)$. To be precise, we have the following Berry-Essen type bound (Theorem 2.4 in \cite{DJS22})
\begin{prop}\label{BerryEssen}
    Suppose $\mathbb{E}X=0,\mathrm{Var}(X)=1$ and there exists $\gamma\ge0,\Delta>0$ such that
    \[\vert\kappa_m(X)\vert\le\frac{(m!)^{1+\gamma}}{\Delta^{m-2}},\;m\ge3,\]
    then there exists some constant $C>0$
    \[\sup_{x\in\mathbb{R}}\vert\mathbb{P}(X\le x)-\mathbb{P}(Z\le x)\vert\le\frac{C}{\Delta^{1/(1+2\gamma)}},\]
    where $Z\sim\mathcal{N}(0,1)$.
\end{prop}
\section{Proof of Theorem \ref{MainThm1}}\label{MainProof1}
Before the proof of our main theorem, we emphasize the following observation. By definition we have for any $1\le l\le m$ and any $x_s,\;1\le s\le m\;s\ne l$
\begin{align}\label{Cancellation}
    \int g_{[m]}(x_1,\dots,x_m)\md x_l=0.
\end{align}
In fact, by considering the position of $x_l$ in the product,
\begin{align}
    \int_{\mathbb{T}^d} g_{[m]}(x_1,\dots,x_m)\md x_l=&\int_{\mathbb{T}^d}\sum_{\pi\vdash[m]}(-1)^{\vert\pi\vert-1}(\vert\pi\vert-1)!\prod_{P\in\pi}f_P\md x_l\notag\\
    =&\sum_{\pi\vdash[m]-\{l\}}\bigg(\vert\pi\vert(-1)^{\vert\pi\vert-1}(\vert\pi\vert-1)!+(-1)^{\vert\pi\vert}(\vert\pi\vert)!\bigg)\prod_{P\in\pi}f_P\notag\\
    =&0,\notag
\end{align}
where we have used the definition of marginals
\[\int_{\mathbb{T}^d} f_P\md x_{l}=f_{P-\{l\}}.\]
Rewrite \eqref{gHierarchy} in Fourier modes
\begin{align}\label{FgHierarchy}
    &\partial_t\hat{g}_{[m]}+\sigma\sum_{k=1}^m(2\pi)^2\vert\xi_k\vert^2\hat{g}_{[m]}\notag\\
    =&-\frac{N-m}{N}\sum_{k=1}^m\hat{H}_k\hat{g}_{[m]\cup\{*\}}-\sum_{k=1}^m\sum_{W\subset[m]-\{k\}}\frac{m-1-\vert W\vert}{N}\hat{H}_k\hat{g}_{W\cup\{k,*\}}\hat{g}_{[m]-\{k\}-W}\notag\\
    &-\frac{N-m}{N}\sum_{k=1}^m\sum_{W\subset[m]-\{k\}}\hat{H}_k\hat{g}_{W\cup\{k\}}\hat{g}_{[m]\cup\{*\}-W-\{k\}}\notag\\
    &+\sum_{k=1}^m\sum_{W\subset[m]-\{k\}}\sum_{R\subset[m]-\{k\}-W}\frac{m-1-\vert W\vert-\vert R\vert}{N}\hat{H}_k\hat{g}_{W\cup\{k\}}\hat{g}_{R\cup\{*\}}\hat{g}_{[m]-R-W-\{k\}}\notag\\
    &-\frac{1}{N}\sum_{k,l=1}^m\hat{S}_{k,l}\hat{g}_{[m]}-\frac{1}{N}\sum_{k,l=1,k\ne l}^m\sum_{W\subset[m]-\{k,l\}}\hat{S}_{k,l}\hat{g}_{W\cup\{k\}}\hat{g}_{[m]-\{k\}-W}.
\end{align}
The identity \eqref{Cancellation} implies that for any $1\le l\le m$ and any $\xi_s,\;1\le s\le m\;s\ne l$
\[\hat{g}_m(\xi_s,0;s\ne l)=0.\]
So in \eqref{FgHierarchy} we only need to consider those modes at the frequency without any zero.\\
Notice that we have turned everything into the frequency side. In this way we can better treat \eqref{FgHierarchy} as a perturbation of the heat equation on the torus. With the perturbative point of view in mind, the proof of Theorem \ref{MainThm1} and Theorem \ref{MainThm2} are both three-folds. The first step is to estimate the forcing terms in a static way. Then we may apply Duhamel's formula to the perturbed equation. Finally some bootstrap arguments would ensure our expected size of chaos. 
\subsection{Estimates of the Forcing Term}
Abbreviate \eqref{FgHierarchy} as
\begin{align}\label{Short}
\partial_t\hat{g}_{[m]}+\sigma\sum_{k=1}^m(2\pi)^2\vert\xi_k\vert^2\hat{g}_{[m]}=\sum_{k=1}^m(2\pi i\xi_k)\cdot R_{[m]}^{(k)}.
\end{align}
Thanks to lemma \ref{BoundOp}, by Plancherel's identity,
\begin{align}
    &\bigg\Vert \frac{2\pi i\xi_k}{\vert2\pi\xi_k\vert}\cdot R_{[m]}^{(k)}(\xi_{[m]})\bigg\Vert_{l^2(\mathbb{Z}^{md})}\notag\\
    =&\Bigg\Vert-\frac{N-m}{N}\vert\nabla_k\vert^{-1}H_kg_{[m]\cup\{*\}}+\sum_{W\subset[m]-\{k\}}\frac{m-1-\vert W\vert}{N}\vert\nabla_k\vert^{-1}H_kg_{W\cup\{k,*\}}g_{[m]-\{k\}-W}\notag\\
    &-\frac{N-m}{N}\sum_{W\subset[m]-\{k\}}\vert\nabla_k\vert^{-1}H_kg_{W\cup\{k\}}g_{[m]\cup\{*\}-W-\{k\}}\notag\\
    &+\sum_{W\subset[m]-\{k\}}\sum_{R\subset[m]-\{k\}-W}\frac{m-1-\vert W\vert-\vert R\vert}{N}\vert\nabla_k\vert^{-1}H_kg_{W\cup\{k\}}g_{R\cup\{*\}}g_{[m]-R-W-\{k\}}\notag\\
    &-\frac{1}{N}\sum_{l=1}^m\vert\nabla_k\vert^{-1}S_{k,l}g_{[m]}-\frac{1}{N}\sum_{l\ne k,l=1}^m\sum_{W\subset[m]-\{k,l\}}\vert\nabla_k\vert^{-1}S_{k,l}g_{W\cup\{k\}}g_{[m]-\{k\}-W}\Bigg\Vert_{L^2(\mathbb{T}^{md})}.
\end{align}
By Minkowski's inequality, we split right hand side term by term,
\begin{align}\label{Es3.1}
    &\bigg\Vert \frac{2\pi i\xi_k}{\vert2\pi\xi_k\vert}\cdot R_{[m]}^{(k)}(\xi_{[m]})\bigg\Vert_{l^2(\mathbb{Z}^{md})}\notag\\
    \le&\Vert K\Vert_{L^\infty}\Bigg(\frac{N-m}{N}\Vert g_{[m]\cup\{*\}}\Vert_{L^2}+\sum_{W\subset[m]-\{k\}}\frac{m-1-\vert W\vert}{N}\Vert g_{W\cup\{k,*\}}\Vert_{L^2}\Vert g_{[m]-\{k\}-W}\Vert_{L^2}\notag\\
    &+\frac{N-m}{N}\sum_{W\subset[m]-\{k\}}\Vert g_{W\cup\{k\}}\Vert_{L^2}\Vert g_{[m]\cup\{*\}-W-\{k\}}\Vert_{L^2}\notag\\
    &+\sum_{W\subset[m]-\{k\}}\sum_{R\subset[m]-\{k\}-W}\frac{m-1-\vert W\vert-\vert R\vert}{N}\Vert g_{W\cup\{k\}}\Vert_{L^2}\Vert g_{R\cup\{*\}}\Vert_{L^2}\Vert g_{[m]-R-W-\{k\}}\Vert_{L^2}\notag\\
    &+\frac{m}{N}\Vert g_{[m]}\Vert_{L^2}+\frac{1}{N}\sum_{l\ne k,l=1}^m\sum_{W\subset [m]-\{k,l\}}\Vert g_{W\cup\{k\}}\Vert_{L^2}\Vert g_{[m]-\{k\}-W}\Vert_{L^2}\Bigg),
\end{align}
where we let $\frac{2\pi i\xi_k}{\vert2\pi\xi_k\vert}\cdot R_{[m]}^{(k)}(\xi_{[m]})$ to be zero when $\xi_k=0$.\\
Moreover, by exchageability, all the
\[\bigg\Vert \frac{2\pi i\xi_k}{\vert2\pi\xi_k\vert}\cdot R_{[m]}^{(k)}(\xi_{[m]})\bigg\Vert_{l^2(\mathbb{Z}^{md})}\]
are equal for $k=1,\dots,m$.
\subsection{Duhamel's Formula}
By Duhamel's principle, Equation \eqref{Short} becomes
\[\hat{g}^t_{[m]}-\exp{\left\{-\sigma t\sum_{k=1}^m(2\pi)^2\vert\xi_k\vert^2\right\}}\hat{g}_{[m]}^0=\int_0^t\md s\exp{\left\{-\sigma (t-s)\sum_{k=1}^m(2\pi)^2\vert\xi_k\vert^2\right\}}\sum_{k=1}^m(2\pi i\xi_k)\cdot R^{(k)}_{[m]}.\]
Via Plancherel's identity and Cauchy-Schwarz inequality, we have
\begin{align}
    &\left\Vert g^t_{[m]}-e^{\sigma t\Delta_{[m]}}g_{[m]}^0\right\Vert_{L^2}^2\notag\\
    =&\sum_{\xi_k;k\in[m]}\left\vert\int_0^t\md s\exp{\left\{-\sigma (t-s)\sum_{l=1}^m(2\pi)^2\vert\xi_l\vert^2\right\}}\sum_{l=1}^m(2\pi i\xi_l)\cdot R^{(l)}_{[m]}(s,\xi_{[m]})\right\vert^2\notag\\
    =&\sum_{\xi_k;k\in[m]}\left\vert\int_0^t\md s\exp{\left\{-\frac{\sigma}{2} (t-s)\sum_{l=1}^m(2\pi)^2\vert\xi_l\vert^2\right\}}\left(\sigma\sum_{l=1}^m(2\pi)^2\vert\xi_l\vert^2\right)^{1/2}\right.\notag\\
    &\left.\cdot\exp{\left\{-\frac{\sigma}{2} (t-s)\sum_{l=1}^m(2\pi)^2\vert\xi_l\vert^2\right\}}\left(\sigma\sum_{l=1}^m(2\pi)^2\vert\xi_l\vert^2\right)^{-1/2}\sum_{l=1}^m(2\pi i\xi_l)\cdot  R^{(l)}_{[m]}(s,\xi_{[m]})\right\vert^2\notag\\
    \le&\sum_{\xi_k;k\in[m]}\int_0^t\md s\exp{\left\{-\sigma (t-s)\sum_{l=1}^m(2\pi)^2\vert\xi_l\vert^2\right\}}\left(\sigma\sum_{l=1}^m(2\pi)^2\vert\xi_l\vert^2\right)\notag\\
    &\cdot\int_0^t\md s\exp{\left\{-\sigma(t-s)\sum_{l=1}^m(2\pi)^2\vert\xi_l\vert^2\right\}}\left(\sigma\sum_{l=1}^m(2\pi)^2\vert\xi_l\vert^2\right)^{-1}\left\vert\sum_{l=1}^m(2\pi i\xi_l)\cdot R^{(l)}_{[m]}(s,\xi_{[m]})\right\vert^2.
\end{align}
Notice the fact that
\[\int_0^tCe^{-Cs}\md s\le 1,\forall t>0\]
and that $\xi_l\ne 0$ for all $l=1,2,\dots,m,$
\begin{align}
    &\left\Vert g^t_{[m]}-e^{\sigma t\Delta_{[m]}}g_{[m]}^0\right\Vert_{L^2}^2\notag\\
    \le&\int_0^t\md se^{-\sigma(2\pi)^2 m(t-s)}\sum_{\xi_k;k\in[m]}\left(\sigma\sum_{l=1}^m(2\pi)^2\vert\xi_l\vert^2\right)^{-1}\left\vert\sum_{l=1}^m\vert2\pi\xi_l\vert\frac{2\pi i\xi_l}{\vert2\pi\xi_l\vert}\cdot R^{(l)}_{[m]}(s,\xi_{[m]})\right\vert^2\notag\\
    \le&\frac{1}{\sigma}\int_0^t\md se^{-\sigma(2\pi)^2 m(t-s)}\sum_{\xi_k;k\in[m]}\sum_{l=1}^m\bigg\vert\frac{2\pi i\xi_l}{\vert2\pi\xi_l\vert}\cdot R_{[m]}^{(l)}(s,\xi_{[m]})\bigg\vert^2.
\end{align}
Again by the fact that
\[\int_0^tCe^{-Cs}\md s\le 1,\forall t>0,\]
we have
\begin{align}\label{Es3.2}
   &\left\Vert g^t_{[m]}-e^{\sigma t\Delta_{[m]}}g_{[m]}^0\right\Vert_{L^2}^2\notag\\
    \le&\frac{1}{m(2\pi\sigma)^2}\sum_{l=1}^m\sup_{0\le s\le t}\sum_{\xi_k;k\in[m]}\bigg\vert\frac{2\pi i\xi_l}{\vert2\pi\xi_l\vert}\cdot R_{[m]}^{(l)}(s,\xi_{[m]})\bigg\vert^2.
\end{align}

By symmetry, we may further deduce that
\begin{align}\label{Es3.3}
\left\Vert g^t_{[m]}-e^{\sigma t\Delta_{[m]}}g_{[m]}^0\right\Vert_{L^2(\mathbb{T}^{md})}\le\frac{1}{2\pi\sigma}\sup_{0\le s\le t}\bigg\Vert\frac{2\pi i\xi_1}{\vert2\pi\xi_1\vert}\cdot R_{[m]}^{(1)}(s)\bigg\Vert_{l^2(\mathbb{Z}^{md})}.
\end{align}
\subsection{Bootstrap}
Fix some constant $C_0>0$. Set $\mathbf{H}(t)$
\[\left\Vert g^s_{[m]}-e^{\sigma s\Delta_{[m]}}g_{[m]}^0\right\Vert_{L^2}\le \frac{C_0(m-1)!}{m^2N^{m-1}}\]
for all $s\le t$ and all $m=1,\dots,N$. And $\mathbf{C}(t)$
\[\left\Vert g^s_{[m]}-e^{\sigma s\Delta_{[m]}}g_{[m]}^0\right\Vert_{L^2}\le \frac{C_0(m-1)!}{2m^2N^{m-1}}\]
for all $s\le t$ and all $m=1,\dots,N$. Then the assertions $(b)(c)(d)$ of Proposition \ref{Bootstrap} can be easily checked due to the absolutely continuity of the $L^2$ norm (on the time interval $[0,T]$). Since the initial data is chaotic, we may choose $C=2\Vert \rho_0\Vert_{L^2}$ such that $\Vert e^{\sigma s\Delta_{[m]}}g_{[m]}^0\Vert_{L^2}$ satisfies the same bound as in $\mathbf{C}(t)$ for all $m=1,\dots,N$. Let
\[\gamma_m(t)=\sup_{s\le t}\left\{\frac{m^2 N^{m-1}}{C_0(m-1)!}\left\Vert g^s_{[m]}-e^{\sigma s\Delta_{[m]}}g_{[m]}^0\right\Vert_{L^2}\right\}.\]
Then
\[\mathbf{H}(t)\Leftrightarrow \forall m,\;\gamma_m(t)\le1,\]
\[\mathbf{C}(t)\Leftrightarrow \forall m,\;\gamma_m(t)\le\frac{1}{2}.\]
By Minkowski's inequality, we may infer from \eqref{Es3.1} and \eqref{Es3.3} that
\begin{align}
    &\gamma_m(t)\le\notag\\
    &\frac{m^2N^{m-1}}{C_0(m-1)!}\frac{8\Vert K\Vert_{L^\infty}}{2\pi\sigma}\Bigg(\frac{N-m}{N}\frac{C_0m!}{(m+1)^2N^m}\gamma_{m+1}(t)\notag\\
    &+\sum_{s=0}^{m-2}\frac{m-1-s}{N}\binom{m-1}{s}\frac{C_0(s+1)!}{(s+2)^2N^{s+1}}\gamma_{s+2}(t)\frac{C_0(m-2-s)!}{(m-1-s)^2N^{m-2-s}}\gamma_{m-1-s}(t)\notag\\
    &+\frac{N-m}{N}\sum_{s=0}^{m-1}\binom{m-1}{s}\frac{C_0s!}{(s+1)^2N^s}\gamma_{s+1}(t)\frac{C_0(m-1-s)!}{(m-s)^2N^{m-1-s}}\gamma_{m-s}(t)\notag\\
    &+\sum_{s=0}^{m-2}\sum_{u=0}^{m-2-s}\frac{m-1-s-u}{N}\frac{(m-1)!}{s!u!(m-1-s-u)!}\frac{C_0s!\gamma_{s+1}(t)}{(s+1)^2N^s}\frac{C_0u!\gamma_{u+1}(t)}{(u+1)^2N^u}\frac{C_0(m-2-s-u)!\gamma_{m-1-s-u}(t)}{(m-1-s-u)^2N^{m-2-s-u}}\notag\\
    &+\frac{m}{N}\frac{C_0(m-1)!}{m^2N^{m-1}}+\frac{m-1}{N}\sum_{s=0}^{m-2}\binom{m-2}{s}\frac{C_0s!}{(s+1)^2N^s}\frac{C_0(m-2-s)!}{(m-1-s)^2N^{m-2-s}}\Bigg)
\end{align}
for $m=1,2,\dots,N$. If we assume that $\mathbf{H}(t)$ holds, then
\begin{align}
    &\gamma_m(t)\le\notag\\
    &\frac{4\Vert K\Vert_{L^\infty}}{\pi\sigma}\Bigg(\frac{(N-m)m}{N^2}+\sum_{s=0}^{m-2}\frac{m^2C_0}{(s+2)^2(m-1-s)^2}+\frac{N-m}{N}\sum_{s=0}^{m-1}\frac{m^2C_0}{(s+1)^2(m-s)^2}\notag\\
    &+\sum_{s=0}^{m-2}\sum_{u=0}^{m-2-s}\frac{m^2C_0^2}{(s+1)^2(u+1)^2(m-1-s-u)^2}+\frac{m}{N}+\sum_{s=0}^{m-2}\frac{m^2C_0}{(s+1)^2(m-1-s)^2}\Bigg).
\end{align}
A little calculation implies that
\begin{align}
    &\gamma_m(t)\notag\\
    \le&\frac{4\Vert K\Vert_{L^\infty}}{\pi\sigma}\Bigg(2+C_0\sum_{s=0}^{m-2}\left(\frac{1}{s+2}+\frac{1}{m-1-s}\right)^2+C_0\sum_{s=0}^{m-1}\left(\frac{1}{s+1}+\frac{1}{m-s}\right)^2\notag\\
    &+C_0\sum_{s=0}^{m-2}\left(\frac{1}{s+1}+\frac{1}{m-1-s}\right)^2\notag\\
    &+C_0^2\sum_{s=0}^{m-2}\frac{(m+1)^2}{(s+1)^2(m-s)^2}\sum_{u=0}^{m-2-s}\left(\frac{1}{u+1}+\frac{1}{m-1-s-u}\right)^2\Bigg).\notag\\
\end{align}
Apply Cauchy-Schwarz inequality term by term,
\begin{align}
    &\gamma_m(t)\notag\\
    \le&\frac{8\Vert K\Vert_{L^\infty}}{\pi\sigma}\Bigg(1+C_0\sum_{s=0}^{m-2}\left(\frac{1}{(s+2)^2}+\frac{1}{(m-1-s)^2}\right)+C_0\sum_{s=0}^{m-1}\left(\frac{1}{(s+1)^2}+\frac{1}{(m-s)^2}\right)\notag\\
    &+C_0\sum_{s=0}^{m-2}\left(\frac{1}{(s+1)^2}+\frac{1}{(m-1-s)^2}\right)\notag\\
    &+C_0^2\sum_{s=0}^{m-2}\frac{(m+1)^2}{(s+1)^2(m-s)^2}\sum_{u=0}^{m-2-s}\left(\frac{1}{(u+1)^2}+\frac{1}{(m-1-s-u)^2}\right)\Bigg)\notag\\
    &\le\frac{8\Vert K\Vert_{L^\infty}}{\pi\sigma}\Bigg(1+3\cdot4C_0+16C_0^2\Bigg).
\end{align}
In a word, if $\sigma>6\Vert K\Vert_{L^\infty}(16C_0^2+12C_0+1)$, then $\gamma_m(t)\le1/2$ for all $m$ given $\mathbf{H}(t)$. By Proposition \ref{Bootstrap}, the conclusion $\mathbf{C}(t)$ holds for all $t$. Thanks to the choice of $C_0$, now we complete the proof of Theorem \ref{MainThm1}.
\section{Proof of Theorem \ref{MainThm2}}\label{MainProof2}
The proof of Theorem \ref{MainThm2} is parallel to that of Theorem \ref{MainThm1}. In fact, under this setting, the operators  $\vert\nabla_k\vert^{-1}S_{k,l}:\hat{l}^\infty(\mathbb{T}^{md})\mapsto\hat{l}^\infty(\mathbb{T}^{md})$ and $\vert\nabla_k\vert^{-1}H_k:\hat{l}^\infty(\mathbb{T}^{(m+1)d})\mapsto\hat{l}^\infty(\mathbb{T}^{md})$ are both bounded. For any $P$ with $k,l\in P$,
\begin{align}
    \Vert\vert\nabla_k\vert^{-1}S_{k,l}h_P\Vert_{\hat{l}^\infty}=&\sup_{\xi_\alpha;\alpha\in P,\xi_k\ne 0}\left\vert\frac{1}{\vert2\pi i\xi_k\vert}\hat{H}_k\hat{h}_P(\xi_\alpha;\alpha\in P)\right\vert\notag\\
    \le&\sup_{\xi_\alpha;\alpha\in P,\xi_k\ne0}\left\vert\sum_{\lambda,\eta}\hat{K}(\lambda,\eta)\hat{h}_P(\xi_\alpha,\xi_k-\lambda,\xi_l-\eta;\alpha\in P-\{k,l\})\right\vert\notag\\
    \le&\Vert\hat{K}\Vert_{l^1(\mathbb{Z}^{2d})}\sup_{\xi_\alpha;\alpha\in P}\vert\hat{h}_P(\xi_\alpha;\alpha\in P)\vert\notag\\
    =&\Vert\hat{K}\Vert_{l^1(\mathbb{Z}^{2d})}\Vert h_P\Vert_{\hat{l}^\infty}.
\end{align}
Similarly we have for any $P$ with $k,*\in P$,
\[\Vert\vert\nabla_k\vert^{-1}H_kh_P\Vert_{\hat{l}^\infty}\le\Vert\hat{K}\Vert_{l^1(\mathbb{Z}^{2d})}\Vert h_P\Vert_{\hat{l}^\infty}.\]
Again write the evolution of $g_{[m]}$'s as \eqref{Short}, where the remainder terms $R_{[m]}^{(k)}$'s satisfy
\begin{align}
    &\left\Vert\frac{2\pi i\xi_k}{\vert2\pi\xi_k\vert}\cdot R_{[m]}^{(k)}(\xi_{[m]})\right\Vert_{l^\infty(\mathbb{Z}^{md})}\notag\\
    =&\Bigg\Vert-\frac{N-m}{N}\vert\nabla_k\vert^{-1}H_kg_{[m]\cup\{*\}}+\sum_{W\subset[m]-\{k\}}\frac{m-1-\vert W\vert}{N}\vert\nabla_k\vert^{-1}H_kg_{W\cup\{k,*\}}g_{[m]-\{k\}-W}\notag\\
    &-\frac{N-m}{N}\sum_{W\subset[m]-\{k\}}\vert\nabla_k\vert^{-1}H_kg_{W\cup\{k\}}g_{[m]\cup\{*\}-W-\{k\}}\notag\\
    &+\sum_{W\subset[m]-\{k\}}\sum_{R\subset[m]-\{k\}-W}\frac{m-1-\vert W\vert-\vert R\vert}{N}\vert\nabla_k\vert^{-1}H_kg_{W\cup\{k\}}g_{R\cup\{*\}}g_{[m]-R-W-\{k\}}\notag\\
    &-\frac{1}{N}\sum_{l=1}^m\vert\nabla_k\vert^{-1}S_{k,l}g_{[m]}-\frac{1}{N}\sum_{l\ne k,l=1}^m\sum_{W\subset[m]-\{k,l\}}\vert\nabla_k\vert^{-1}S_{k,l}g_{W\cup\{k\}}g_{[m]-\{k\}-W}\Bigg\Vert_{\hat{l}^\infty(\mathbb{Z}^{md})}.
\end{align}
By Minkowski's inequality,
\begin{align}\label{Es4.1}
     &\left\Vert\frac{2\pi i\xi_k}{\vert2\pi\xi_k\vert}\cdot R_{[m]}^{(k)}(\xi_{[m]})\right\Vert_{l^\infty(\mathbb{Z}^{md})}\notag\\
    \le&\Vert\hat{K}\Vert_{l^1}\Bigg(\frac{N-m}{N}\Vert g_{[m]\cup\{*\}}\Vert_{\hat{l}^\infty}+\sum_{W\subset[m]-\{k\}}\frac{m-1-\vert W\vert}{N}\Vert g_{W\cup\{k,*\}}\Vert_{\hat{l}^\infty}\Vert g_{[m]-\{k\}-W}\Vert_{\hat{l}^\infty}\notag\\
    &+\frac{N-m}{N}\sum_{W\subset[m]-\{k\}}\Vert g_{W\cup\{k\}}\Vert_{\hat{l}^\infty}\Vert g_{[m]\cup\{*\}-W-\{k\}}\notag\\
    &+\sum_{W\subset[m]-\{k\}}\sum_{R\subset[m]-\{k\}-W}\frac{m-1-\vert W\vert-\vert R\vert}{N}\Vert g_{W\cup\{k\}}\Vert_{\hat{l}^\infty}\Vert g_{R\cup\{*\}}\Vert_{\hat{l}^\infty}\Vert g_{[m]-R-W-\{k\}}\Vert_{\hat{l}^\infty}\notag\\
    &+\frac{m}{N}\Vert g_{[m]}\Vert_{\hat{l}^\infty}+\frac{1}{N}\sum_{l\ne k,l=1}^m\sum_{W\subset[m]-\{k,l\}}\Vert g_{W\cup\{k\}}\Vert_{\hat{l}^\infty}\Vert g_{[m]-\{k\}-W}\Vert_{\hat{l}^\infty}\Bigg).
\end{align}
Notice that the upper bound are equal for all $k=1,\dots,m$. By Duhamel's principle,
\[\hat{g}^t_{[m]}-\exp{\left\{-\sigma t\sum_{k=1}^m(2\pi)^2\vert\xi_k\vert^2\right\}}\hat{g}_{[m]}^0=\int_0^t\md s\exp{\left\{-\sigma (t-s)\sum_{k=1}^m(2\pi)^2\vert\xi_k\vert^2\right\}}\sum_{k=1}^m(2\pi i\xi_k)\cdot R^{(k)}_{[m]}.\]
Take $\hat{l}^\infty$ norm on both sides,
\begin{align}
    &\Vert g_{[m]}^t-e^{\sigma t\Delta_{[m]}}g_{[m]}^0\Vert_{\hat{l}^\infty}\notag\\
    =&\sup_{\xi_k;k\in[m]}\Bigg\vert\int_0^t\md s\exp{\left\{-\sigma (t-s)\sum_{l=1}^m(2\pi)^2\vert\xi_l\vert^2\right\}}\sum_{l=1}^m(2\pi i\xi_l)\cdot R^{(l)}_{[m]}\Bigg\vert\notag\\
    \le&\sup_{s\le t}\left\Vert\frac{2\pi i\xi_k}{\vert2\pi\xi_k\vert}\cdot R_{[m]}^{(k)}(s,\xi_{[m]})\right\Vert_{l^\infty(\mathbb{Z}^{md})}\sup_{\xi_k;k\in[m]}\Bigg\vert\int_0^t\md s\exp{\left\{-\sigma (t-s)\sum_{l=1}^m(2\pi)^2\vert\xi_l\vert^2\right\}}\sum_{l=1}^m\vert2\pi i\xi_l\vert\Bigg\vert.
\end{align}
By Cauchy-Schwarz inequality,
\begin{align}
    &\Vert g_{[m]}^t-e^{\sigma t\Delta_{[m]}}g_{[m]}^0\Vert_{\hat{l}^\infty}\notag\\
    \le&\sup_{s\le t}\left\Vert\frac{2\pi i\xi_k}{\vert2\pi\xi_k\vert}\cdot R_{[m]}^{(k)}(s,\xi_{[m]})\right\Vert_{l^\infty(\mathbb{Z}^{md})}\notag\\
    &\cdot\sup_{\xi_k;k\in[m]}\Bigg\vert\int_0^t\md s\exp{\left\{-\sigma (t-s)\sum_{l=1}^m(2\pi)^2\vert\xi_l\vert^2\right\}}(\sum_{l=1}^m1)^{1/2}(\sum_{l=1}^m\vert2\pi\xi_l\vert^2)^{1/2}\Bigg\vert\notag\\
    \le&\frac{1}{2\pi\sigma}\sup_{s\le t}\left\Vert\frac{2\pi i\xi_k}{\vert2\pi\xi_k\vert}\cdot R_{[m]}^{(k)}(s,\xi_{[m]})\right\Vert_{l^\infty(\mathbb{Z}^{md})}.
\end{align}
Substitute the $\hat{l}^\infty$ estimate of forcing terms, and we have
\begin{align}\label{EsInf}
    &\Vert g_{[m]}^t-e^{\sigma t\Delta_{[m]}}g_{[m]}^0\Vert_{\hat{l}^\infty}\notag\\
    \le&\frac{\Vert\hat{K}\Vert_{l^1}}{2\pi\sigma}\sup_{s\le t}\Bigg(\frac{N-m}{N}\Vert g^s_{[m]\cup\{*\}}\Vert_{\hat{l}^\infty}+\sum_{W\subset[m]-\{k\}}\frac{m-1-\vert W\vert}{N}\Vert g^s_{W\cup\{k,*\}}\Vert_{\hat{l}^\infty}\Vert g^s_{[m]-\{k\}-W}\Vert_{\hat{l}^\infty}\notag\\
    &+\frac{N-m}{N}\sum_{W\subset[m]-\{k\}}\Vert g^s_{W\cup\{k\}}\Vert_{\hat{l}^\infty}\Vert g^s_{[m]\cup\{*\}-W-\{k\}}\Vert_{\hat{l}^\infty}\notag\\
    &+\sum_{W\subset[m]-\{k\}}\sum_{R\subset[m]-\{k\}-W}\frac{m-1-\vert W\vert-\vert R\vert}{N}\Vert g^s_{W\cup\{k\}}\Vert_{\hat{l}^\infty}\Vert g^s_{R\cup\{*\}}\Vert_{\hat{l}^\infty}\Vert g^s_{[m]-R-W-\{k\}}\Vert_{\hat{l}^\infty}\notag\\
    &+\frac{m}{N}\Vert g^s_{[m]}\Vert_{\hat{l}^\infty}+\frac{1}{N}\sum_{l\ne k,l=1}^m\sum_{W\subset[m]-\{k,l\}}\Vert g^s_{W\cup\{k\}}\Vert_{\hat{l}^\infty}\Vert g^s_{[m]-\{k\}-W}\Vert_{\hat{l}^\infty}\Bigg)
\end{align}
for $m=1,2,\dots,N$. Recall that $g_{[1]}^0=\rho_0$ is a probability density, thus it is non-negative and has total mass $1$. So all the Fourier modes of $g_{[1]}^0$ satisfy
\[\vert\hat{g}_{[1]}^0(k)\vert=\left\vert\int e^{2\pi ik\cdot x}\rho_0(x)\md x\right\vert\le\int\rho_0(x)\md x=1\]
for all $k\in\mathbb{Z}^d$, we may conduct the following bootstrap argument.\\
Set $\mathbf{H}(t)$
\[\left\Vert g_{[m]}^s-e^{\sigma s\Delta_{[m]}}g_{[m]}^0\right\Vert_{\hat{l}^\infty}\le\frac{2(m-1)!}{m^2N^{m-1}}\]
for all $s\le t$ and all $m=1,\dots,N$.\\
And $\mathbf{C}(t)$
\[\left\Vert g_{[m]}^s-e^{\sigma s\Delta_{[m]}}g_{[m]}^0\right\Vert_{\hat{l}^\infty}\le\frac{(m-1)!}{m^2N^{m-1}}\]
for all $s\le t$ and all $m=1,\dots,N$.\\
Since the initial data is chaotic, we have that $\Vert e^{\sigma s\Delta_{[m]}}\Vert_{\hat{l}^\infty}$ also satisfies the same bound as in $\mathbf{C}(t)$ for all $m=1,\dots,N$. Let
\[\delta_m(t)=\sup_{s\le t}\Bigg\{\frac{m^2N^{m-1}}{(m-1)!}\left\Vert g_{[m]}^s-e^{\sigma s\Delta_{[m]}}g_{[m]}^0\right\Vert_{\hat{l}^\infty}\Bigg\}.\]
Then
\[\mathbf{H}(t)\Leftrightarrow\forall m,\;\delta_m(t)\le2,\]
\[\mathbf{C}(t)\Leftrightarrow\forall m,\;\delta_m(t)\le1.\]
With these notations, the estimate \eqref{EsInf} becomes
\begin{align}
    &\delta_m(t)\le\notag\\
    &\frac{m^2N^{m-1}}{(m-1)!}\frac{8\Vert \hat{K}\Vert_{l^1}}{2\pi\sigma}\Bigg(\frac{N-m}{N}\frac{m!}{(m+1)^2N^m}\delta_{m+1}(t)\notag\\
    &+\sum_{s=0}^{m-2}\frac{m-1-s}{N}\binom{m-1}{s}\frac{(s+1)!}{(s+2)^2N^{s+1}}\delta_{s+2}(t)\frac{(m-2-s)!}{(m-1-s)^2N^{m-2-s}}\delta_{m-1-s}(t)\notag\\
    &+\frac{N-m}{N}\sum_{s=0}^{m-1}\binom{m-1}{s}\frac{s!}{(s+1)^2N^s}\delta_{s+1}(t)\frac{(m-1-s)!}{(m-s)^2N^{m-1-s}}\delta_{m-s}(t)\notag\\
    &+\sum_{s=0}^{m-2}\sum_{u=0}^{m-2-s}\frac{m-1-s-u}{N}\frac{(m-1)!}{s!u!(m-1-s-u)!}\frac{s!\delta_{s+1}(t)}{(s+1)^2N^s}\frac{u!\delta_{u+1}(t)}{(u+1)^2N^u}\frac{(m-2-s-u)!\delta_{m-1-s-u}(t)}{(m-1-s-u)^2N^{m-2-s-u}}\notag\\
    &+\frac{m}{N}\frac{(m-1)!}{m^2N^{m-1}}\delta_m(t)+\frac{m-1}{N}\sum_{s=0}^{m-2}\binom{m-2}{s}\frac{s!}{(s+1)^2N^s}\delta_{s+1}(t)\frac{(m-2-s)!}{(m-1-s)^2N^{m-2-s}}\delta_{m-1-s}(t)\Bigg).
\end{align}
If we assume $\mathbf{H}(t)$ holds, then
\begin{align}
    &\delta_m(t)\le\notag\\
    &\frac{4\Vert \hat{K}\Vert_{l^1}}{\pi\sigma}\Bigg(\frac{2(N-m)m}{N^2}+\sum_{s=0}^{m-2}\frac{2^2m^2}{(s+2)^2(m-1-s)^2}+\frac{N-m}{N}\sum_{s=0}^{m-1}\frac{2^2m^2}{(s+1)^2(m-s)^2}\notag\\
    &+\sum_{s=0}^{m-2}\sum_{u=0}^{m-2-s}\frac{2^3m^2}{(s+1)^2(u+1)^2(m-1-s-u)^2}+\frac{2m}{N}+\sum_{s=0}^{m-2}\frac{2^2m^2}{(s+1)^2(m-1-s)^2}\Bigg).\notag\\
\end{align}
A little calculation yields
\begin{align}
    &\delta_m(t)\notag\\
    \le&\frac{4\Vert \hat{K}\Vert_{l^1}}{\pi\sigma}\Bigg(2+4\sum_{s=0}^{m-2}\left(\frac{1}{s+2}+\frac{1}{m-1-s}\right)^2+4\sum_{s=0}^{m-1}\left(\frac{1}{s+1}+\frac{1}{m-s}\right)^2\notag\\
    &+4\sum_{s=0}^{m-2}\left(\frac{1}{s+1}+\frac{1}{m-1-s}\right)^2\notag\\
    &+8\sum_{s=0}^{m-2}\frac{(m+1)^2}{(s+1)^2(m-s)^2}\sum_{u=0}^{m-2-s}\left(\frac{1}{u+1}+\frac{1}{m-1-s-u}\right)^2\Bigg).
\end{align}
Apply Cauchy-Schwarz inequality,
\begin{align}
    &\delta_m(t)\le\notag\\
    &\frac{8\Vert \hat{K}\Vert_{l^1}}{\pi\sigma}\Bigg(1+4\sum_{s=0}^{m-2}\left(\frac{1}{(s+2)^2}+\frac{1}{(m-1-s)^2}\right)+4\sum_{s=0}^{m-1}\left(\frac{1}{(s+1)^2}+\frac{1}{(m-s)^2}\right)\notag\\
    &+4\sum_{s=0}^{m-2}\left(\frac{1}{(s+1)^2}+\frac{1}{(m-1-s)^2}\right)\notag\\
    &+8\sum_{s=0}^{m-2}\frac{(m+1)^2}{(s+1)^2(m-s)^2}\sum_{u=0}^{m-2-s}\left(\frac{1}{(u+1)^2}+\frac{1}{(m-1-s-u)^2}\right)\Bigg)\notag\\
    &\le\frac{8\Vert \hat{K}\Vert_{l^1}}{\pi\sigma}\Bigg(1+3\cdot4\cdot 4+16\cdot8\Bigg)\le\frac{600\Vert\hat{K}\Vert_{l^1}}{\sigma}.
\end{align}
Thus there exists a universal constant $C\le 600$ such that, if $\sigma>C\Vert\hat{K}\Vert_{l^1}$ then $\mathbf{H}(t)$ implies $\mathbf{C}(t)$, which completes the bootstrap argument. That is $\mathbf{C}(t)$ holds for all $t\ge0$. Thus
\[\Vert g_{[m]}^t\Vert_{\hat{l}^\infty}\le\frac{2(m-1)!}{m^2N^{m-1}}\]
for all $m=1,\dots,N$ and all $t\ge0$.
\section{Fluctuation around Mean-Field Limit}\label{Discussion}
We adopt the assumptions of Theorem \ref{MainThm2} in this section. The goal is to illustrate how to infer the collective behavior of the particles from the estimates of the size of chaos. That is, to study the mean-field limit as well as the fluctuation around it base on these estimates. Let us emphasize the first two equations in \eqref{gHierarchy}, namely,
\begin{align}
    \partial_tg_{1,N}-\sigma\Delta g_{1,N}=&-\frac{N-1}{N}\mathrm{div}\left(\int K(x,y)g_{1,N}(x)g_{1,N}(y)\md y\right)\notag\\&-\frac{N-1}{N}\mathrm{div}\left(\int K(x,y)g_{2,N}(x,y)\md y\right)-\frac{1}{N}\mathrm{div}\Big(K(x,x)g_{1,N}(x)\Big)
\end{align}
and
\begin{align}
    \partial_tg_{2,N}-\sigma(\Delta_x+\Delta_y)g_{2,N}=&-\frac{N-2}{N}\Bigg(\mathrm{div}_x\left(\int K(x,z)g_{1,N}(x)g_{2,N}(y,z)\md z\right)\notag\\
    &\qquad\qquad+\mathrm{div}_y\left(\int K(y,z)g_{1,N}(y)g_{2,N}(x,z)\md z\right)\Bigg)\notag\\
    &-\frac{1}{N}\Bigg(\mathrm{div}_x\left(\int K(x,z)g_{1,N}^{\otimes3}\md z\right)+\mathrm{div}_y\left(\int K(y,z)g_{1,N}^{\otimes3}\md z\right)\Bigg)\notag\\
    &-\frac{1}{N}\Bigg(\mathrm{div}_x\left(K(x,y)g_{1,N}^{\otimes2}\right)+\mathrm{div}_y\left(K(y,x)g_{1,N}^{\otimes2}\right)\Bigg)\notag\\
    &+\frac{1}{N}r_N,
\end{align}
where $r_N=O(\frac{1}{N})$ in view of Theorem \ref{MainThm2}. Due to the similarity between Equation \eqref{McKeanVlasov} and the equation for $g_{1,N}$, we may expect that $g_{1,N}(t)$ converges to $\rho_t$. Consequently, we may also expect that $Ng_{2,N}(t)$ converges to $b_t$, where $b_t(x,y)$ satisfies
\begin{align}
    &\partial_tb-\sigma(\Delta_x+\Delta_y)b(y,z)\notag\\
    =&-\mathrm{div}_x\left(\int K(x,z)\rho(x)b(y,z)\md z\right)+\mathrm{div}_y\left(\int K(y,z)\rho(y)b(x,z)\md z\right)\notag\\
    &-\mathrm{div}_x\left(\int K(x,z)\rho^{\otimes3}\md z\right)-\mathrm{div}_y\left(\int K(y,z)\rho^{\otimes3}\md z\right)\notag\\
    &-\mathrm{div}_x\left(K(x,y)\rho^{\otimes2}\right)-\mathrm{div}_y\left(K(y,x)\rho^{\otimes2}\right).
\end{align}
with $b_0=0$, at least under $\hat{l}^\infty$ norm. See \cite{B46,D21} for its origin. We shall prove these limits in Subsection \ref{Correction} first.
\subsection{Bogolyubov Correction}\label{Correction}
\begin{prop}\label{MFL}
    Under the assumption of Theorem \ref{MainThm2}, there exists a constant $C=C(K)$ such that for $\sigma>C$ we have
    \[\Vert g_{1,N}-\rho\Vert_{\hat{l}^\infty}\le\frac{C_1}{N}\]
    uniformly in $t$, where $C_1$ is a universal constant.
\end{prop}
\begin{rmk}
    Proposition \ref{MFL} actually implies propagation of chaos. For fixed $j$, we have
    \begin{align}
        \Vert f_{j,N}-\rho^{\otimes j}\Vert_{\hat{l}^\infty}=&\Vert\sum_{\pi\vdash[j]}\prod_{P\in\pi}g_{P,N}-\rho^{\otimes j}\Vert_{\hat{l}^\infty}\notag\\
        \le&\Vert g_{1,N}^{\otimes j}-\rho^{\otimes j}\Vert_{\hat{l}^\infty}+\Vert\sum_{\pi\vdash[j],\vert\pi\vert<j}\prod_{P\in\pi}g_{P,N}\Vert_{\hat{l}^\infty}\notag\\
        \le&\frac{C_j}{N}
    \end{align}
    by Theorem \ref{MainThm2} and Proposition \ref{MFL}, where $C_j$ depends only on $j$. Recall the discussions before Theorem \ref{MainThm2}, the bound guarantees propagation of chaos.
\end{rmk}
\begin{proof}
    Consider the evolution of $D^t(\xi):=\hat{g}^t_{1,N}(\xi)-\hat{\rho}_t(\xi)$, which reads
    \begin{align}
        \partial_tD+\sigma(2\pi)^2\vert\xi\vert^2D=&-\hat{H}_{1}D_{1}\hat{g}_{\{*\}}-\hat{H}_1\hat{\rho}D_{\{*\}}-\hat{H}_1D_{[1]}D_{\{*\}}\notag\\
        &-\frac{N-1}{N}\hat{H}_1\hat{g}_{[1]\cup\{*\}}-\frac{1}{N}\hat{S}_{1,1}\hat{g}_{[1]}-\frac{1}{N}\hat{H}_1\hat{g}_{[1]}\hat{g}_{\{*\}}.
    \end{align}
    Recall that $D^0=0$, by Duhamel's formula we have
    \begin{align}
    D^t=&\int_0^t\md s\exp{\Big\{-\sigma(t-s)(2\pi)^2\vert\xi\vert^2\Big\}}\Bigg(-\hat{H}_{1}D_{1}\hat{g}_{\{*\}}-\hat{H}_1\hat{\rho}D_{\{*\}}-\hat{H}_1D_{[1]}D_{\{*\}}\notag\\
        &-\frac{N-1}{N}\hat{H}_1\hat{g}_{[1]\cup\{*\}}-\frac{1}{N}\hat{S}_{1,1}\hat{g}_{[1]}-\frac{1}{N}\hat{H}_1\hat{g}_{[1]}\hat{g}_{\{*\}}\Bigg).
    \end{align}
    Argue as in Section \ref{MainProof2}, and we have
    \begin{align}
        \Vert D^t\Vert_{l^\infty}\le&\frac{\Vert K\Vert_{l^1}}{2\pi\sigma}\sup_{s\le t}\Bigg(2\Vert D^s\Vert_{l^\infty}+\Vert D^s\Vert_{l^\infty}^2+\frac{N-1}{N}\Vert g_{2,N}\Vert_{\hat{l}^\infty}+\frac{2}{N}\Bigg).
    \end{align}
    Since $\Vert g_{1,N}\Vert_{\hat{l}^\infty}=\Vert\rho\Vert_{\hat{l}^\infty}=1$. Since $D^0=0$, a bootstrap trick verifies the validity of the same constant $C(K)$ as in Theorem \ref{MainThm2}.
\end{proof}
The following proposition can be derived via similar arguments. 
\begin{prop}\label{BCorrection}
    Under the assumption of Theorem \ref{MainThm2}, there exists a constant $C=C(K)$ such that for $\sigma>C$ we have
    \[\Vert Ng_{2,N}-b\Vert_{\hat{l}^\infty}\le\frac{C_2}{N}.\]
    uniformly in $t$, where $C_2$ is a universal constant.
\end{prop}
\subsection{Empirical Cumulants}
To illustrate the validity of our $\hat{l}^\infty$ framework, we shall prove CLT for the empirical measure associated with the particle dynamics \eqref{ParticleSystem} in this and the next subsection. To apply Theorem \ref{BerryEssen}, we need to estimate the empirical cumulants. In fact, the following proposition holds.
\begin{prop}\label{preBound}
    Under the assumptions of Theorem \ref{MainThm2}, there exists a constant $C=C(K)$ such that for $\sigma>C$, for any test function $\phi\in C^\infty(\mathbb{T}^d)$ and any $m=1,2,\dots$
    \[\Bigg\vert\kappa_{m}\left[\int\phi\md\mu_N^t\right]\Bigg\vert\le\frac{\left(8\Vert\hat{\phi}\Vert_{l^1}\right)^m(m!)^4}{N^{m-1}}.\]
\end{prop}
\begin{proof}
Notice that
\[\kappa_{m}\left[\int\phi\md\mu_N^t\right]=\sum_{\pi\vdash[m]}(-1)^{\vert\pi\vert-1}(\vert\pi\vert-1)!\prod_{B\in\pi}\mathbb{E}\left[\left(\int\phi\md\mu_N^t\right)^{\vert B\vert}\right]\]
and that
\begin{align}
    \mathbb{E}\left[\left(\int\phi\md\mu_N^t\right)^j\right]=&\frac{1}{N^j}\sum_{l_1,\dots,l_j=1}^m\mathbb{E}^N\left[\prod_{k=1}^n\phi(X^t_{j_k})\right]\notag\\
    =&\frac{1}{N^j}\sum_{\pi\vdash[j]}N(N-1)\cdots(N-\vert\pi\vert+1)\int\left(\bigotimes_{B\in\pi}\phi^{\vert B\vert}\right)f_{\vert\pi\vert,N}.
\end{align}
Moreover,
\[f_{j,N}=\sum_{\pi\vdash[j]}\prod_{P\in\pi}g_{P,N}.\]
Thus we may express cumulants in terms of correlation functions
\begin{align}
    \kappa_m\left[\int\phi\md\mu_N^t\right]=\sum_{\pi\vdash[m]}N^{\vert\pi\vert-m}\sum_{\rho\vdash\pi}K_N(\rho)\int\left(\bigotimes_{B\in\pi}\phi^{\vert B\vert}\right)\left(\bigotimes_{P\in\rho}g_{\vert P\vert,N}\right)\md x_1\dots\md x_{\vert\pi\vert},
\end{align}
where the variables of $\phi$ and $g_{\vert P\vert,N}$ are determined by the partitions $\pi$ and $\rho$ and the constant $K_N(\rho)$ is given by
\[K_N(\rho)=\sum_{\iota\ge\rho}(-1)^{\vert\iota\vert-1}(\vert\iota\vert-1)!\prod_{C\in\iota}\Bigg(\left(1-\frac{1}{N}\right)\cdots\left(1-\frac{\vert C\vert-1}{N}\right)\Bigg).\]
The following lemma provides an upper bound of $K_N(\rho)$. 
\begin{lemma}\label{KNBound}
    For all $m=1,2,\dots$ and all $\rho\vdash[m]$,
    \[\vert K_N(\rho)\vert\le m!N^{1-\vert\rho\vert}.\]
\end{lemma}
We leave the proof of Lemma \ref{KNBound} in Appendix \ref{aLemma}. Given that Lemma \ref{KNBound} holds, some crude bounds are enough for the proof. 
By Plancherel' identity, we have
\begin{align}
    \Bigg\vert\kappa_{m}\left[\int\phi\md\mu_N^t\right]\Bigg\vert=&\left\vert\sum_{\pi\vdash[m]}N^{\vert\pi\vert-m}\sum_{\rho\vdash\pi}K_N(\rho)\int\left(\bigotimes_{B\in\pi}\phi^{\vert B\vert}\right)\left(\bigotimes_{P\in\rho}g_{\vert P\vert,N}\right)\md x_1\dots\md x_{\vert\pi\vert}\right\vert\notag\\
    =&\left\vert\sum_{\pi\vdash[m]}N^{\vert\pi\vert-m}\sum_{\rho\vdash\pi}K_N(\rho)\sum\left(\bigotimes_{B\in\pi}\phi^{\vert B\vert}\right)^\wedge\left(\overline{\bigotimes_{P\in\rho}g_{\vert P\vert,N}}\right)^\wedge\right\vert\notag\\
    \le&\sum_{\pi\vdash[m]}N^{\vert\pi\vert-m}\sum_{\rho\vdash\pi}\vert K_N(\rho)\vert\prod_{B\in\pi}\left\Vert\left(\phi^{\vert B\vert}\right)^\wedge\right\Vert_{l^1}\prod_{P\in\rho}\left\Vert g_{\vert P\vert,N}\right\Vert_{\hat{l}^\infty}.
\end{align}
Thanks to Young's inequality, for any $l=1,2,\dots$,
\[\Vert(\phi^l)^{\wedge}\Vert_{l^1}\le\Vert\hat{\phi}\Vert^l_{l^1}.\]
Thus we have
\begin{align}
    &\Bigg\vert\kappa_{m}\left[\int\phi\md\mu_N^t\right]\Bigg\vert\notag\\
    \le&\sum_{\pi\vdash[m]}N^{\vert\pi\vert-m}\sum_{\rho\vdash\pi}\vert\pi\vert!N^{1-\vert\rho\vert}\Vert\hat{\phi}\Vert_{l^1}^m\prod_{P\in\rho}\frac{2(\vert P\vert-1)!}{\vert P\vert^2N^{\vert P\vert-1}}\notag\\
    =&\Vert\hat{\phi}\Vert_{l^1}^mN^{1-m}\sum_{\pi\vdash[m]}\vert\pi\vert!\sum_{\rho\vdash\pi}2^{\vert\rho\vert}\prod_{P\in\rho}\frac{(\vert P\vert-1)!}{\vert P\vert^2}.
\end{align}
Notice that the number of partitions of an $L$-element set is at most $2^{L-1}L!$. In fact, we have $2^{L-1}$ ways to divide $L$ ordered vacant positions into groups and $L!$ ways to fill these positions with all the $L$ elements.\\
Take $L=\vert\pi\vert$ first for each $\pi\vdash[m]$. For $\rho\vdash\pi$, we have $\vert\rho\vert\le \vert\pi\vert$. Moreover, if we write the factorial $(\vert P\vert-1)!$ out for each $P\in\rho$ and find proper factor in $\vert\pi\vert!$ to compare with, we may estimate the product by
\[\prod_{P\in\rho}\frac{(\vert P\vert-1)!}{\vert P\vert^2}\le\vert\pi\vert!.\]
Thus
\begin{align}
    \Bigg\vert\kappa_{m}\left[\int\phi\md\mu_N^t\right]\Bigg\vert
    \le\Vert\hat{\phi}\Vert_{l^1}^mN^{1-m}\sum_{\pi\vdash[m]}(\vert\pi\vert!)^32^{\vert\pi\vert-1}2^{\vert\pi\vert}.
\end{align}
Then take $L=m$ and notice that $\vert\pi\vert\le m$ for all $\pi\vdash[m]$, we have
\begin{align}
    \Bigg\vert\kappa_{m}\left[\int\phi\md\mu_N^t\right]\Bigg\vert
    \le\Vert\hat{\phi}\Vert_{l^1}^mN^{1-m}(m!)^34^m(m!)2^{m-1}
    \le(m!)^4\left(8\Vert\hat{\phi}\Vert_{l^1}\right)^mN^{1-m}.
\end{align}
\end{proof}
\subsection{Proof of Corollary \ref{CLT}}
By virtue of Proposition \ref{BerryEssen} and Proposition \ref{preBound}, the first assertion of Corollary \ref{CLT} is a consequence of the second assertion. Now we compute
\begin{align}
    \sigma^2_{\phi,N}=&\mathbb{E}\left(\int\phi\md\mu_N^t\right)^2-\left(\mathbb{E}\int\phi\md\mu_N^t\right)^2\notag\\
    =&\frac{1}{N}\int\phi^2\md f_{1,N}+\frac{N-1}{N}\int\phi\otimes\phi\md f_{2,N}-\left(\int\phi\md f_{1,N}\right)^2\notag\\
    =&\frac{1}{N}\int\phi^2\md g_{1,N}-\frac{1}{N}\left(\int\phi\md g_{1,N}\right)^2+\frac{N-1}{N}\int\phi\otimes\phi\md g_{2,N}.
\end{align}
Thanks to Plancherel's identity, we may apply Proposition \ref{MFL} and \ref{BCorrection} to conclude.
\subsection*{Acknowledgements}
The author thanks Zhenfu Wang for inspiring discussions. This work was partially supported by the National Key R\&D Program of China, Project Number 2021YFA1002800 and NSFC grant No.12031009.
\begin{appendices}
   \section{Proof of Lemma \ref{KNBound}}\label{aLemma}
    The assertion of Lemma \ref{KNBound} is obvious when $\vert\rho\vert=1,2$. For $\rho\vdash[m]$ with $\vert\rho\vert\ge3$, consider the polynomial
    \[K(x,\rho):=\sum_{\iota\ge\rho}(-1)^{\vert\iota\vert-1}(\vert\iota\vert-1)!\prod_{C\in\iota}\bigg(\Big(1-x\Big)\dots\Big(1-(\vert C\vert-1)x\Big)\bigg).\]
    Then we have that $K_N(\rho)=K(1/N,\rho)$. For any polynomial $f(x)$, denote by $C_l(f)$ the coefficient of $x^l$ in $f(x)$($C_l(f)=0$ when $l>\mathrm{deg}(f)$). We claim that $C_{l}(K(x,\rho))=0$ for $l<\vert\rho\vert-1$ and that for all $m=1,2,\dots$
    \[\sum_{l=\vert\rho\vert-1}^m\left\vert C_l\Big(K(x,\rho)\Big)\right\vert\le m!.\]
    It is easy to see that once the claim is proved, the proposition is automatically true. Now verify the vanishing of the coefficients by induction. Suppose that we have checked all the $\rho'\vdash[m]$ for $m\le k$, consider the partition $\rho$ of $[k+1]$. We distinguish two cases, the singleton $\{k+1\}\notin\rho$ and $\{k+1\}\in\rho$.\\
    In the former case, there exist a nonempty set $A\subset[k]$ such that $A\cup\{k+1\}\in\rho$. For any $B\in\rho$, define $\psi(B)=B$ if $k+1\notin B$ and $=A$ otherwise. Then $\rho'=\psi(\rho)\vdash[k]$ and $\psi$ is a bijection from $\rho$ to $\rho'$. Thus we may calculate for $l<\vert\rho\vert-1$,
    \begin{align}
        &C_l\Big(K(x,\rho)\Big)\notag\\
        =&\sum_{\iota\ge\rho}(-1)^{\vert\iota\vert-1}(\vert\iota\vert-1)!C_l\Bigg(\prod_{C\in\iota}\bigg(\Big(1-x\Big)\dots\Big(1-(\vert\psi( C)\vert-1)x\Big)\bigg)\Bigg)\notag\\
        &-\sum_{\iota\ge\rho}(-1)^{\vert\iota\vert-1}(\vert\iota\vert-1)!C_{l-1}\Bigg(\prod_{C\in\iota}\bigg(\Big(1-x\Big)\dots\Big(1-(\vert\psi( C)\vert-1)x\Big)\bigg)\Bigg)\Bigg(\sum_{Q:\exists P\in\iota,Q\cup A\subset P}\vert \psi(Q)\vert\Bigg)\notag\\
        =&C_l\Big(K(x,\rho')\Big)-\vert A\vert C_{l-1}\Big(K(x,\rho')\Big)\notag\\
        &-\sum_{\iota\ge\rho}(-1)^{\vert\iota\vert-1}(\vert\iota\vert-1)!C_{l-1}\Bigg(\prod_{C\in\iota}\bigg(\Big(1-x\Big)\dots\Big(1-(\vert\psi( C)\vert-1)x\Big)\bigg)\Bigg)\Bigg(\sum_{\substack{Q\ne A\cup\{k+1\}\\\exists P\in\iota,Q\cup A\subset P}}\vert Q\vert\Bigg).
    \end{align}
    By the Induction Hypothesis, the first two terms vanishes. We change the order of the sum and find that
    \begin{align}
        &C_l\Big(K(x,\rho)\Big)\notag\\
        =&0-\sum_{Q\in\rho,Q\ne A\cup\{k+1\}}\vert Q\vert\sum_{\iota\ge\rho(Q)}(-1)^{\vert\iota\vert-1}(\vert\iota\vert-1)!C_{l-1}\Bigg(\prod_{C\in\iota}\bigg(\Big(1-x\Big)\dots\Big(1-(\vert C\vert-1)x\Big)\bigg)\Bigg)\notag\\
        =&-\sum_{Q\in\rho,Q\ne A\cup\{k+1\}}\vert Q\vert C_{l-1}\Bigg(K(x,\rho(Q))\Bigg)\notag\\
        =&0,
    \end{align}
    where $\rho(Q)$ is the minimal element among all the partitions of $[k]$ such that $\rho(Q)\ge\rho'$ and that $Q\cup A\in \rho(Q)$. Thus $\vert\rho(Q)\vert=\vert\rho\vert-1\ge2$ and $l-1<\vert\rho(Q)\vert-1$, ensuring the last ``$=$'' above thanks to the Induction Hypothesis.\\
    In the latter case, denote by $\rho'=\rho\backslash\{\{k+1\}\}$. We separate the sum into two parts according to whether $\{k+1\}$ is in $\iota$. Then we have
    \begin{align}
        &C_l\Big(K(x,\rho)\Big)\notag\\
        =&\sum_{\iota\ge\rho,\{k+1\}\notin\iota}(-1)^{\vert\iota\vert-1}(\vert\iota\vert-1)!C_l\Bigg(\prod_{C\in\iota}\bigg(\Big(1-x\Big)\dots\Big(1-(\vert C\vert-1)x\Big)\bigg)\Bigg)\notag\\
        &+\sum_{\iota\ge\rho,\{k+1\}\in\iota}(-1)^{\vert\iota\vert-1}(\vert\iota\vert-1)!C_l\Bigg(\prod_{C\in\iota}\bigg(\Big(1-x\Big)\dots\Big(1-(\vert C\vert-1)x\Big)\bigg)\Bigg).
    \end{align}
Notice that for each $\iota\ge\rho$ with $\{k+1\}\notin\iota$, we may regard it as a way to add $k+1$ into some elements of certain $\iota'\ge\rho'$. While for each $\iota\ge\rho$ with $\{k+1\}\in\iota$, we may find a unique $\iota'\ge\rho'$ such that $\iota=\iota'\cup\{\{k+1\}\}$. Sum according to these $\iota'$'s,
    \begin{align}
        &C_l\Big(K(x,\rho)\Big)\notag\\
        =&\sum_{\iota'\ge\rho'}(-1)^{\vert\iota'\vert-1}(\vert\iota'\vert-1)!C_l\Bigg(\prod_{C\in\iota'}\bigg(\Big(1-x\Big)\dots\Big(1-(\vert C\vert-1)x\Big)\bigg)\bigg(\sum_{C\in\iota'}\Big(1-\vert C\vert x\Big)-\vert\iota'\vert\bigg)\Bigg)\notag\\
        =&C_l(kxK(x,\rho'))\notag\\
        =&kC_{l-1}(K(x,\rho'))\notag\\
        =&0.
    \end{align}
     The last ``$=$'' is due to the Induction Hypothesis and the fact that $\vert\rho'\vert=\vert\rho\vert-1>l-1$.\\
    Thus the first half of our claim is verified.\\
    As for the last half, We can also carry out induction. The bound in the claim is true for $\vert\rho\vert=1,2$. Suppose we have got the bound for $\rho\vdash[m]$ with $m\le k$. Then for the set $[k+1]$, we distinguish two cases, the singleton $\{k+1\}\notin\rho$ and $\{k+1\}\in\rho$, as well.\\
    In the former case, with the same notation, we have the bound
    \[
        \left\vert C_l\Big(K(x,\rho)\Big)\right\vert\le\left\vert C_l\Big(K(x,\rho')\Big)\right\vert+\vert A\vert\left\vert C_{l-1}\Big(x,\rho'\Big)\right\vert+\sum_{Q\in\rho,Q\ne A\cup\{k+1\}}\vert Q\vert\left\vert C_{l-1}\Big(K_{l-1}(x,\rho(Q))\Big)\right\vert.
    \]
    In the latter case,
    \[
        \left\vert C_l\Big(K(x,\rho)\Big)\right\vert\le k\left\vert C_{l-1}\Big(K(x,\rho)\Big)\right\vert.
    \]
    Summing over $l$ and we find that in both case
    \[\sum_{l=0}^\infty\left\vert C_l\Big(K(x,\rho)\Big)\right\vert\le (k+1)\sup_{m\le k}\sup_{\iota\vdash[m]}\sum_{l=0}^\infty\left\vert C_l\Big(K(x,\iota)\Big)\right\vert,\]
    which guarantees the bound.
\end{appendices}
\bibliographystyle{amsalpha}
\bibliography{ref}
\end{document}